\documentclass[10pt, a4paper]{amsart}
\usepackage{hyperref}
\usepackage{amssymb, amsmath, amsthm, amsfonts}
\usepackage{bbm}
\usepackage{enumerate}
\usepackage{dsfont}
\usepackage[mathscr]{eucal}
\usepackage{tikz}
\usepackage{color}
\usepackage{setspace}
\definecolor{gray}{gray}{0.5}

\theoremstyle{plain}
   \newtheorem{theorem}{Theorem}[section]

   \newtheorem{lemma}[theorem]{Lemma}
   \newtheorem{claim}[theorem]{Claim}
   \newtheorem{definition}[theorem]{Definition}

    \numberwithin{equation}{section}

  \newcommand{\di}{\mathrm{dim}}
  \newcommand{\off}{\text{off}}
  \newcommand{\mc}{\mathcal}
\newcommand{\ch}{\mathrm{ch}}

\newcommand{\FF}{\mathcal{F}}

\newcommand{\HH}{\mathcal{H}}

\renewcommand{\bar}{\overline}

\newcommand{\N}{\mathbb{N}} 
\newcommand{\Q}{\mathbb{Q}} 
\newcommand{\R}{\mathbb{R}} 

\def\scriptO{{{\it O}\kern -.42em {\it `}\kern + .20em}}

\begin{document}
\title[Mid-point convex and convex]{On Convexity, Mid-point Convexity and Hausdorff measures of Sets}
\author[S. Guo \  \  \  \  \ \ T. Lan \ \ \ \ \ \ Y. Xi] {Shaoming Guo, \ \ Tian Lan, \ \ Yakun Xi}
\address{Shaoming Guo: Department of Mathematics, University of Wisconsin-Madison, 480 Lincoln Dr, Madison, WI, 53706, USA}
\email{shaomingguo2018@gmail.com}

\address{Tian Lan: Department of Mathematics, The Chinese University of Hong Kong, Shatin, N.T., Hong Kong}
\email{1155091994@link.cuhk.edu.hk}

\address{Yakun Xi: Department of Mathematics, University of Rochester, 500 Joseph C Wilson Blvd, Rochester, NY 14627, USA}
\email{yxi4@math.rochester.edu}

\begin{abstract}
We give a complete characterization of the size of Borel sets that are mid-point convex but not (essentially) convex, in terms of their Hausdorff dimensions and Hausdorff measures. 
\end{abstract}

\date{\today}

\maketitle

\section{introduction}

Fix a dimension $d\ge 1$. Let $E\subset \R^d$ be a (Lebesgue) measurable set. We say that $E$ is mid-point convex if $\frac{x+y}{2}$ belongs to $E$ whenever $x, y\in E$. We say that $E$ is convex if $tx+(1-t)y$ belongs to $E$ for every $t\in [0, 1]$, whenever $x, y\in E$. It is a standard exercise in measure theory that if $d=1$, $E$ is mid-point convex and $|E|>0$, then $E$ is convex.

Such a result can also be generalized to higher dimensions. We first need to introduce
\begin{definition}
A set $E\subset \R^d$ is called {\bf essentially convex} if there exist an integer $s\in [1, d]$ and a translation of a linear subspace of dimension $s$, which will be called $L_s$, such that $E\subset L_s$, and under the induced topology on $L_s$, it holds that $E^\mathrm{o}\neq \emptyset$ and $\ch(E)\setminus E\subset \partial(\ch(E))$. Here $\ch(E)$ refers to the convex hull of $E$. 
\end{definition}
Now we can state that, if $E\subset \R^d$ is mid-point convex and $|E|>0$, then $E$ is essentially convex. \\

 In the current paper, we study the case where the condition $|E|>0$ fails. The question we are interested in is: If $E$ is a \emph{Borel} set that is mid-point convex but not (essentially) convex, how ``large" can it be? If we use Hausdorff dimension and Hausdorff measure to measure ``size", then we can provide a complete answer to the above question.

\begin{theorem}\label{structure_thm}
\begin{itemize}
    \item[1)] Let $s\in (0, d)$. Let $E$ be a Borel set that is mid-point convex but not essentially convex. Then $\mc{H}^s(E)=0$ or $\infty$. Moreover, there exist Borel sets $E_1$ and $E_2$ of dimension $s$ that are mid-point convex but not essentially convex satisfying $\mc{H}^s(E_1)=0$ and $\mc{H}^s(E_2)=\infty$. 
    \item[2)] Let $s=d$. There exists a Borel set $E$ of dimension $s$ that is mid-point convex but not essentially convex satisfying $\mc{H}^s(E)=0$.
\end{itemize}
\end{theorem}

In Section \ref{sectiontwo} we will prove that every set $E$ of Hausdorff dimension $s$ with $0<s<d$ that is mid-point convex but not essentially convex must have an $s$-dimensional Hausdorff measure $0$ or $\infty$. In Section \ref{sectionthree} we will construct one-dimensional Borel sets $E_1$ and $E_2$ with $\mc{H}^1(E_1)=0$ and $\mc{H}^1(E_2)=\infty$ that are mid-point convex but not essentially convex. This section consists of the main difficulty of the paper. Our constructions are variants of those in \cite{SS10}. In Section \ref{sectionfour}, we will construct desired sets in Theorem \ref{structure_thm} by using the above $E_1$ and $E_2$ on $\R$. \\

%

{\bf Notation:} Let $B\subset \R^d$. We will use $\di_{\mathcal H}(B)$ to denote the Hausdorff dimension of the set $B$. Moreover, $\mathcal H^{\alpha}(B)$ denotes the $\alpha$-dimensional Hausdorff measure of the set $B$. When $\alpha=d$, we write $|B|:=\mc{H}^{\alpha}(B)$. For a real number $x\in \R$, we use $[x]$ to denote the floor function of $x$. Moreover, $\{x\}$ is used to denote the decimal part of $x$, that is, $\{x\}=x-[x]$.  Throughout the paper, we will use $\mc{D}$ to denote the collection of all dyadic numbers. \\

{\bf Acknowledgements:} The authors would like to thank Po-Lam Yung for a number of insightful discussions. The authors would like to thank Nikolaos Chatzikonstantinou for bringing the paper \cite{SS10} to our attention. The first author was supported in part by a direct grant for research from
the Chinese University of Hong Kong (4053295), and by NSF grant 1800274. The second author was partially supported by grant CUHK24300915 from the Hong Kong Research Grants Council. The third author was supported in part by the AMS-Simons travel grant.


%
%

\section{Rigidity in mid-point convexity}\label{sectiontwo}

Let $s\in (0, d)$. Let $E$ be a Borel set with $0< \mathcal{H}^s(E)<\infty$ that is mid-point convex. We will first show that $s$ must be an integer and then show that $E$ must be essentially convex. This will finish the proof of one part of Theorem \ref{structure_thm} that $\mc{H}^s(E)=0$ or $\infty$, under the assumption that $E$ is mid-point convex but not essentially convex. 
 
 Let us first recall some basic definitions.  Given $x\in \R^d$, define the upper density of $E$ at $x$ by 
\begin{equation}
\overline D^s(E, x)=\overline{\lim\limits_{r\rightarrow 0}}\frac{\mathcal H^s(E\cap B_r(x))}{(2r)^s}.
\end{equation}
Here $B_r(x)$ denotes the closed ball of radius $r$ centered at $x$. Then we have (Proposition 5.1 in \cite{Fal14})
\begin{equation}\label{181101e2.2}
\overline D^s(E, x)=0 \text{ for $ \mathcal{H}^s$-almost all } x\not\in E,
\end{equation}
and 
\begin{equation}
2^{-s}\le \overline D^s(E, x)\le 1 \text{ for $ \mathcal{H}^s$-almost all } x\in E.
\end{equation}
 Therefore there exists a point $x_0\in E$ such that the upper density of $E$ at $x_0$ satisfies $2^{-s}\le \overline D^s(E,x_0)\le 1$. 
 \begin{lemma}\label{xy}
 For every $x,\ y\in E$, $ t\in[0,1]$, if we assume that $\overline D^s(E, y)>0,$ then we have $$\overline D^s(E, tx+(1-t)y)>0,$$
 for all $t\in[0,1].$
\end{lemma}
Let us first see how we can use Lemma \ref{xy} to obtain the desired estimate on $\mc{H}^s(E)$. First of all, for every $x\in \text{ch}(E)$, we have $\overline D^s(E, x)>0$. Therefore by \eqref{181101e2.2} we see that 
\begin{equation}\label{chE}
     0<\mathcal H^s(\text{ch}(E))=\mathcal H^s(E)<\infty.
\end{equation}
We conclude that $\ch(E)$ must have Hausdorff dimension $s$, thus $s$ is an integer. Next, we will show that $E$ is essentially convex. Up to a rotation of $E$, we can without loss of generality assume that $\ch(E)\subset \R^s$, where $\R^s$ is a coordinate space in $\R^d$. Now it is a standard exercise to apply the Lebesgue density theorem in $\R^s$ and conclude that $E$ is essentially convex. \\
%

It remains to prove Lemma \ref{xy}.
\begin{proof}[Proof of Lemma \ref{xy}]  Without loss of generality, we assume that $x$ is the origin and $y=(1, 0, \dots, 0)$. For simplicity, we denote $D=\overline D^s(E, y)>0.$ By definition, we can find a small positive number $r_0$ such that 
\[\frac{\mathcal H^s(E\cap B_{r_0}(y))}{(2r_0)^s}\ge \frac12D.\]
First we prove

%
\begin{claim}\label{claim11}
 For every dyadic number $k2^{-n}$ with $n\ge 0$ and $ 0<k\le2^n$ being odd, we have
\[\frac{\mathcal H^s(E\cap B_{2^{-n}r_0}(y_{k, n}))}{(2^{1-n}r_0)^s}\ge \frac12 D,\]
where $y_{k, n}=(k2^{-n}, 0, \dots, 0)$.
\end{claim}
The claim follows by an induction on $n\ge0$. The base case $n=0$ is trivial. If $n=1$, it suffices to show that
\[{\mathcal H^s(E\cap B_{\frac12 r_0}(y_{1, 1}))}\ge \frac12Dr_0^s.\]
Since $E$ is midpoint convex, $x, y\in E$, we have $$E\supset E/2\supset \frac12 (E\cap B_{ r_0}(y))= (E/2)\cap B_{ \frac12r_0}(y_{1, 1}).$$
Therefore 
\[{\mathcal H^s(E\cap B_{\frac12 r_0}(y_{1, 1}))}\ge{\mathcal H^s((E/2)\cap B_{ \frac12r_0}(y_{1, 1}))}=2^{-s}{\mathcal H^s(E\cap B_{ r_0}(y))}\ge \frac12 D{r_0^s}.\]
By the homogeneity of $\mathcal H^s$, the same proof works for every $n>0$ by induction. This finishes the proof of the claim. 

It remains to show that
$\overline D^s(E, ty)>0$ for every $t\in[0,1)$. We fix $t\in [0, 1)$. For every sufficiently small $r>0$, denote $n:=-[\log_2{r}]+2$. This choice of $n$ guarantees that  $\frac18r\le 2^{-n}\le\frac14 r$, and that there exists an odd integer $k$ such that $B_{2^{-n}r_0}(y_{k, n})\subset B_{2^{-n}}(y_{k, n})\subset B_r(ty).$
Therefore, by Claim \ref{claim11}, we obtain 
	\[\mathcal H^s(E\cap B_r(ty))\ge\mathcal H^s(E\cap B_{2^{-n}r_0}(y_{k, n}))\ge 2^{-1} D(2^{1-n}r_0)^s\ge2^{-1-3s}Dr_0^s(2r)^s,\]
which further implies that
\[\frac{\mathcal H^s(E\cap B_r(ty))}{(2r)^s}\ge 2^{-1-3s}Dr_0^s>0.	\]
This finishes the proof of Lemma \ref{xy}. \end{proof}

\section{Constructing one dimensional sets}\label{sectionthree}
Given a compact set $A\subset [0,1)$ and a positive integer $j$, we denote 
\begin{equation}
    A_j:= \{a_1+a_2+\dots+a_j \,|\,a_1,\,a_2,\,\dots a_j\in A\}.
\end{equation}
We now prove the following theorem.
\begin{theorem}\label{mainthm}
Let $\{\alpha_j\}_{j=1}^{\infty}$ be a non-decreasing sequence in $(0,1)$. Then there exists a compact set $A\subset [0,1)$ such that $\di_{\mathcal H} (A_j)=\alpha_j$ for every $j\ge 1$. Moreover, we can construct such an $A$ satisfying $\mathcal H^{\alpha_1}(A)=\infty$.
\end{theorem}

The same construction also works if we weaken our assumption to $\{\alpha_j\}_{j=1}^{\infty}\subset [0, 1]$. Here to avoid certain technical issues, we choose to avoid the values $0$ and $1$. 

Before we present the proof of Theorem \ref{mainthm}, let us first see how we can use it to prove Theorem \ref{structure_thm} in the one dimensional case.
\begin{proof}[Proof of Theorem \ref{structure_thm}: Case $d=1$.] Firstly, given $s\in (0, 1]$, we construct a set $E$ with dimension $s$ and with $\mathcal H^s(E)=0$ that is mid-point convex but not convex. Take an increasing sequence with limit $s$, say $\{\frac{js}{j+1}\}_{j=1}^{\infty}$. By Theorem \ref{mainthm}, we know that there exists a set $A$ such that each $A_j$ has dimension $\frac{js}{j+1}$. Define 
	\begin{equation}\label{defi_set}
	E:=\bigcup_{j\ge 1}\frac{A_j}{j}.
	\end{equation}
	It is not difficult to see that $E$ is mid-point convex. Moreover, $\mathcal H^{s}(E)=0$ and $\mathcal H^{s-\epsilon}(E)=\infty$ for every $\epsilon>0$. This finishes the construction of the desired set. \\
	
	Next for $s\in(0,1)$, we construct a set $E$ with dimension $s$ and $\mathcal H^{s}(E)=\infty$. We will take the sequence in the above theorem to be the constant sequence $\{s\}_{j=1}^{\infty}$. Theorem \ref{mainthm} says that there exists $A$ such that $\di_{\mc{H}}(A_j)=s$ and $\mathcal H^s(A)=\infty$. Again we define $E$ as in \eqref{defi_set}. This will produce the desired set. 
\end{proof}

Before proving  Theorem \ref{mainthm}, let us provide some basic definitions. For each $x\in[0,1)$, we will associate to it a sequence $\underline x:=x_1x_2\dots$ such that $x=\sum\limits_{i=1}^{\infty}\frac{x_i}{2^i}$, where $x_1, \,x_2,\,\dots\in\{0,1\}$. If $x$ has two distinct such representations, we call it a $\textit{dyadic rational}$. In this case, we will use $\underline{x}=x_1x_2\dots$
to denote the sequence that contains infinitely many zeros.
We will denote by $\mc{D}$ the set of all such numbers. Notice that $\mc{D}$ is countable, and therefore  a set of the form $A\setminus \mc{D}$  has the same Hausdorff dimension as $A$.
\begin{definition}[$n$-cell]
For $a_1, a_2, \dots a_n\in \{0,1\}$, we will denote the closed interval of length $2^{-n}$ starting from $\sum\limits_{i=1}^{n}\frac{a_i}{2^i}$ by $[a_1a_2\dots a_n]$, and call it an \textit{n-cell}.
\end{definition}

For a Borel set $B\subset [0,1)$, we will construct a Borel probability measure supported on $B$ and use the following lemma due to Billingsley (see e.g. Lemma 1.4.1 in \cite{CY}) to find a lower bound for the Hausdorff dimension of $B$. 

\begin{lemma}[Billingsley's lemma]
Let $B\subset[0,1]$ be Borel and let $\mu$ be a
finite Borel measure on $[0,1]$. Suppose $\mu(B)> 0$. Let $I_n(x)$ be the unique $n$-cell containing $x\in[0,1]\setminus\mc{D}$. If
\[\alpha\le \liminf\limits_{n\rightarrow\infty}\frac{\log\;\mu(I_n(x))}{\log|I_n(x)|}\le\beta\]
for all $x\in B\setminus\mc{D}$, then $\alpha\le\di_\HH(B)\le\beta$.
\end{lemma}

For $\underline x=x_1x_2\dots \in B\setminus \mc{D}$ and every $n\ge 0$, we define 
\begin{equation}
    \#_{\off}(n,\underline{x},B):=
    \begin{cases}
    1& \text{if}\;[x_1x_2\dots x_n 0]\cap (B\setminus \mc{D}) \neq \emptyset\; \text{and} \;[x_1x_2\dots x_n 1]\cap (B\setminus \mc{D})\neq \emptyset\\
    0 & \text{otherwise}
    \end{cases}
    \end{equation}
For every $n\ge 1$, we define \begin{gather}
    \mathcal O\FF\FF_n(B):=\min\limits_{\underline x\in B\setminus \mc{D}}\;\frac 1n\sum\limits^{n-1}_{i=0}\#_{\off}(i,\underline x,B).
\end{gather}
\begin{lemma}\label{lemmahauest}
For a compact set $B\subset [0,1)$, it holds that
\begin{equation}\label{hausest}
       \liminf\limits_{n\rightarrow\infty}\mathcal O\FF\FF_n(B)\le \di_{\mathcal H} B. 
\end{equation}
\end{lemma}

This lemma is essentially Lemma 2 from \cite{SS10}. Here we impose an extra assumption on the set $B$ that it is compact. Without this assumption, the statement \eqref{hausest} may not be entirely correct. For instance, one can take $B=\Q\cap [0, 1]$, so that the left hand side of \eqref{hausest} is $1$, while the right hand side is $0$.

\begin{proof}[Proof of Lemma \ref{lemmahauest}.]
Without loss of generality we can assume $B\setminus \mc{D}\neq\emptyset$. Firstly we construct a Borel probability measure on $B_0:=B\setminus \mc{D}$. Define $\mu(B_0):=1$, and then define it on the algebra generated by half-open dyadic intervals by induction: Let $C$ be a dyadic interval $(k2^{-n},(k+1)2^{-n}]$ with $\mu(C\cap B_0)=t$. If both $C_1:=(k2^{-n},(2k+1)2^{-n-1}]\cap B_0$ and $C_2:=((2k+1)2^{-n-1},(k+1)2^{-n}]\cap B_0$ are non-empty, then define $\mu(C_1)=\mu(C_2)=\frac 12 t$. Otherwise define $\mu(C_i)= t$ if $C_i$ is non-empty, $0$ if empty $(i\,\in\{0,1\})$. Then it is not difficult to verify that $\mu$ is a pre-measure on the algebra generated by the intersection of dyadic intervals with $B_0$. Caratheodory's extension theorem and Caratheodory's criterion allow us to extend $\mu$ to a Borel probability measure on $B_0$, and then we  extend it to a Borel probability measure on $[0,1]$ by setting $\mu(X):=\mu(X\cap B_0)$. For any $\underline x:=x_1x_2\dots\in B_0$, let $I_n(x)$ denote the unique $n$-cell containing $x$. By induction on $n$, we see that 
$$
\mu(I_n(x))=2^{-\sum\limits_{i=0}^{n-1}\#_{off}(i,\underline x,B)}\le 2^{-n\cdot \mathcal O\FF\FF_n(B)}.
$$
Hence $$\liminf\limits_{n\rightarrow\infty}\frac{\log\;\mu(I_n(x))}{\log|I_n(x)|}\ge\liminf\limits_{n\rightarrow\infty}\mathcal O\FF\FF_n(B).$$ By using Billingsley's lemma, we have the inequality \eqref{hausest}.
\end{proof}

This lemma can be used to calculate the Hausdorff dimension of a certain class of sets. 
\begin{lemma} \label{calhaudim}
Let $S$ be a subset of $\N^+$. Let $A:=\{x\in[0,1): x_j=0\;\;\forall j\notin S\}$. Then
\begin{equation}
\di_{\mathcal H} (A)=\mathrm{den}(S).
\end{equation}
Here $\mathrm{den}(S)$ is the lower density of the set $S$ defined by
\[\mathrm{den}(S):=\liminf\limits_{n\to\infty} \frac{|\{1,2,\dots n\}\cap S|}{n}.\]
\begin{proof}[Proof of Lemma \ref{calhaudim}.]
First of all, it is not difficult to see that $A$ is closed, therefore compact. 
By Lemma \ref{lemmahauest}, we have 
\begin{equation}
\di_{\mathcal H} (A)\ge \liminf\limits_{n\rightarrow\infty}\mathcal O\FF\FF_n(A)=\text{den}(S).
\end{equation}
On the other hand, for every $\epsilon>0$, there exists infinitely many $n_1<n_2<\cdots$ such that $\frac{|\{1,2,\dots n_i\}\cap S|}{n_i}<\text{den}(S)+\epsilon$  for each $i$. To compute the Hausdorff dimension of $A$, we will first cover $A$ by almost disjoint closed $n_i$-cells. Note that there are at most $2^{|\{1,2,\dots n_i\}\cap S|}\le 2^{n_i(\text{den}(S)+\epsilon)}$ many such cells that intersect $A$. Therefore by the definition of the Hausdorff measure, we have 
\begin{equation}
    \mathcal H^{\text{den}(S)+2\epsilon}_{2^{-n_i}}(A)\le 2^{n_i(\text{den}(S)+\epsilon)}\,2^{-n_i(\text{den}(S)+2\epsilon)}.
\end{equation}
Taking a limit $i\to \infty$, one sees that 
    $\mathcal H^{\text{den}(S)+2\epsilon}(A)=0,$
for every $\epsilon>0$. Therefore, we can conclude that $\di_{\mathcal H}(A)\le \text{den}(S)$. This finishes the proof of the lemma. 
\end{proof}
\end{lemma}
\begin{lemma}\label{massdisprin}
Let $A$ be a Borel subset of $\R^d$. Suppose there exists $\delta >0$ and a Borel probability measure $\mu$ supported on $A$ such that $\mu(I)\le c(|I|)|I|^{\delta}$ for every measurable set $I$, with $\lim\limits_{|I|\rightarrow 0} c(|I|)=0$, then $\mathcal H^{\delta}(A)=\infty$.
\end{lemma}
\begin{proof}[Proof of Lemma \ref{massdisprin}.]
The proof is essentially the same as the proof of the mass distribution principle. For any $\epsilon>0$, there exists $\delta_0>0$ such that $c(x)<\epsilon$ for any $0<x\le \delta_0$. For any $t<\delta_0$, suppose $\{J_i\}_{i=1}^{\infty}$ covers $A$ with $|J_i|\le t$ (without loss of generality we can assume each $J_i$ is compact), then 
\begin{equation}
 1=\mu(A)\le \sum\limits_{i=1}^{\infty}\mu(J_i)\le\sum\limits_{i=1}^{\infty}c(|J_i|)\, |J_i|^\delta \le\sum\limits_{i=1}^{\infty}\epsilon\, |J_i|^\delta.
 \end{equation}
In other words, we have $\sum\limits_{i=1}^{\infty}|J_i|^\delta\ge \frac 1\epsilon$. Hence $\mathcal H^\delta_t(A)\ge \frac 1\epsilon$. Let $t\rightarrow 0$ we obtain $\mathcal H^\delta(A)\ge \frac 1\epsilon$. Since $\epsilon$ is arbitrary, we have $\mathcal H^\delta(A)=\infty$.
\end{proof}
The following lemma shows that we can control the Hausdorff dimension of finitely many $A_i$.
\begin{lemma} \label{finitecase}
Given $i\ge 1$. Let $0< \alpha_1\le\alpha_2\le\dots \alpha_i< 1$. Then there exists a non-empty compact set $A^i\subset [0,1]$ such that $\di_{\mathcal H} A^i_j=\alpha_j$ for every $1\le j\le i$. Moreover, one can pick such an $A^i$ satisfying $\mathcal H^{\alpha_1}(A^i)=\infty$. Here $A^i_j$ is defined as $\underbrace{A^i+A^i+\cdots A^i}_j$.
\end{lemma}
\begin{proof}[Proof of Lemma \ref{finitecase}.]
We decompose $\{0,1,\dots \frac{i(i+1)}2-1\}$ into $i$ disjoint sets: \begin{equation}
\begin{split}
    & Z^i_1=\{a^i_{1,1},a^i_{1,2}\}:=\{0,1\},\,Z^i_2=\{a^i_{2,1},a^i_{2,2},a^i_{2,3}\}:=\{2,3,4\},\dots,\\
    &Z^i_{i-1}=\{a^i_{i-1,1},a^i_{i-1,2},\dots, a^i_{i-1,i}\}:=\{\frac {(i+1)(i-2)}2,\dots, \frac{i(i+1)}2-2\},\\
    & Z^i_i=\{a^i_{i,1}\}:= \{\frac{i(i+1)}2-1\}.
\end{split}
    \end{equation}
Then we choose a rapidly ``increasing" sequence
of intervals $\{[\gamma^i_s,\eta^i_s]\}_{s=1}^{\infty}$ inductively with length
$d^i_s:=\eta^i_s-\gamma^i_s+1$, $\gamma^i_1>2$ and \begin{equation}\label{incresing_intervals_1}
\eta_{s+1}^i\ge\gamma^i_{s+1}>2^{2^s}\eta^i_s\ge
2^{2^s}\sum\limits_{n=1}^sd^i_n.   
\end{equation} 
Since
$\alpha_1<1$, we can choose $\gamma^i_s$ large enough so that
\begin{equation}\label{incresing_intervals_2}
2^{-(\gamma^i_s-1-\sum\limits_{n=1}^{s-1}d^i_n))}<\frac 1s\cdot
2^{-\alpha_1\gamma^i_s}.    
\end{equation}
Let $t$ be the unique number depending on $s$ such that
there exist integers $k$ and $q$ such that $s=\frac {i(i+1)k}2+a^i_{t,q}$. 
We take \begin{gather}\label{incresing_intervals_3}
\eta^i_s:= \max\Big\{\gamma^i_s,\bigg [ \frac{\gamma^i_s-1-\sum\limits_{n=1}^{s-1}d^i_n-\log_2 s}{\alpha_t}\bigg ]\Big\}.
\end{gather}
Here $[x]$ is the floor function of $x$. This finishes the definition of the above sequence of intervals. \\

For $1\le l \le i $, we define $\mathcal{B}^i_l$ to be the set
\begin{equation}
\begin{split}
    & \mathbb \bigcup_{\substack{k\ge 1, t\le i-1,\\ q\in \{1,2\dots, t+1\}\setminus{\{l\}}}} \big [\gamma^i_{\frac{i(i+1)k}2+a^i_{t, q}}\;,\,\eta^i_{\frac{i(i+1)k}2+a^i_{t, q}}\big ] \bigcup_{k\ge 1,k\in\N^+}\big[\gamma^i_{\frac{i(i+1)k}2+a^i_{i, 1}}\;,\,\eta^i_{\frac{i(i+1)k}2+a^i_{i, 1}}\big].
\end{split}
\end{equation}
Next, we define
\begin{equation}
    B^i_l :=\{x\in[0,1):x_s=0\;\;\text{if}\;\; s\in \mathcal{B}^i_l\;\},
\end{equation}
and
\begin{equation}
    A^i:= \bigcup_{1\le l\le i}B^i_l.
\end{equation}
Note that $A^i$ is compact. 
For every $1\le j\le i$, and $1\le l_1,l_2\dots l_j\le i$, we claim that there exists $q$ such that  
$B^i_{l_1}+B^i_{l_2}+\cdots+ B^i_{l_j}$ is contained in 
\begin{equation}\label{sum_i_many}
\{x\in [0,j]:y=\{x\}, y_s=0 \textup{ whenever }s\in \bigcup_{k\ge
1}\;[\gamma^i_{\frac{i(i+1)k}2+a^i_{j, q}}\;,\,\eta^i_{\frac{i(i+1)k}2+a^i_{j, q}}-2j\;]\}.
\end{equation}
When $j=i$, the statement is immediate, and we have no choice but choosing $q=1$. In the case $j<i$, the statement follows from the fact that $j$ sets of the form
\begin{equation}
    \{1,2,\dots j+1\}\setminus \{l_{j'}\}, \ j'=1, 2, \dots, j,
\end{equation}must share at least one common element. \\

Now we are ready to apply Lemma \ref{calhaudim} to compute the Hausdorff dimension of the set \eqref{sum_i_many}. Recall the setup in \eqref{incresing_intervals_1}--\eqref{incresing_intervals_3}. In particular, we have 
\begin{equation}\label{191003e2.18}
    \lim_{s\to \infty} \frac{\sum\limits_{n=1}^sd^i_n}{\eta_{s+1}}=0,
\end{equation}
and 
\begin{equation}\label{191003e2.19}
\lim_{k\to \infty} \gamma^i_{\frac{i(i+1)k}2+a^i_{j, q}}\,\big / (\eta^i_{\frac{i(i+1)k}2+a^i_{j, q}}-2j)=\alpha_j.
\end{equation}
%
Therefore, we apply Lemma \ref{calhaudim} and obtain that $\di_{\mathcal H} (A^i_j)\le \alpha_j$. \\

On the other hand,
$A^i_j$ contains $ B^i_1+B^i_2+\cdots +B^i_j$, which further contains 
\begin{equation}
\begin{split}
&\Big\{x\in [0,1):x_s=0\;\;\text{whenever}\;\; s\in \bigcup_{\substack{k\ge 1, t\ge j,\\q}}
\;[\gamma^i_{\frac{i(i+1)k}2+a^i_{t, q}}\;,\,\eta^i_{\frac{i(i+1)k}2+a^i_{t, q}}\;]\Big\}
\end{split}
\end{equation}
By combining Lemma \ref{calhaudim} with estimates \eqref{191003e2.18} and \eqref{191003e2.19}, we obtain $\di_{\mathcal H} (A^i_j)\ge \alpha_j$. This finishes the proof that $\di_{\mathcal H} (A^i_j)=\alpha_j$.\\

The remaining part is to estimate $\mathcal H^{\alpha_1}(A^i)$. To achieve the goal, we construct a Borel
probability measure $\mu^i$ supported on $B_1^i$ which is defined via induction. We start with $\mu^i([0, 1))=1.$ Let $D\subset [0, 1)$ be a left open and right closed dyadic interval. Denote $\mu_D:=\mu(D)$. Without loss of generality, $\mu_D>0$. Write $D=D_L\cup D_R$, where $D_L$ and $D_R$ are both left open and right closed dyadic intervals of length $|D|/2$.  If $(D_L\setminus \mc{D})\cap B^i_1\neq \emptyset$ and $(D_R\setminus \mc{D})\cap B^i_1\neq \emptyset$, then we set $\mu(D_L)=\mu(D_R)=\mu_D/2$. Otherwise, we assign the full measure of $D$ to whichever set between $D_L$ and $D_R$ that has non-empty intersection with $B^i_1\setminus \mc{D}$. This finishes the definition of the pre-measure supported on $B^i_1$. Since $B^i_1$ is compact, it follows that the measure $\mu^i$ extended by this pre-measure is also supported on $B^i_1$.\\

We claim that 
\begin{equation}\label{191004e4.21}
\mu^i(I)\le \frac 1s |I|^{\alpha_1}
\end{equation} 
holds for every $s \ge 1$ and every $n$-cell $I$ with $n\ge \gamma^i_s$. This claim, combined with Lemma \ref{massdisprin}, will imply that 
\[
\mathcal H^{\alpha_1}(A^i)\ge \mathcal H^{\alpha_1}(B_1^i)=\infty.\]
To see this, it suffices to consider
$n=\eta_s^i$ for the same $s$. If $\gamma_s^i=\eta_s^i$, we have \begin{equation}
\mu^i(I)\le 2^{-(\gamma^i_s-1-\sum\limits_{n=1}^{s-1}d^i_n))}<\frac 1s\cdot
2^{-\alpha_1\gamma^i_s}=\frac 1s |I|^{\alpha_1}.
\end{equation} 
Otherwise, we have 
\begin{equation}
\mu^i(I)\le 2^{-(\gamma^i_s-1-\sum\limits_{n=1}^{s-1}d^i_n))}\le \frac 1s\cdot
2^{-\alpha_t\eta^i_s}\le \frac 1s\cdot 2^{-\alpha_1\eta^i_s}=\frac 1s |I|^{\alpha_1}
\end{equation}
by the definition of $\eta_s^i$, where $t$ is defined as in the line below equation \eqref{incresing_intervals_2}. This finishes the proof of \eqref{191004e4.21}.
\end{proof}

%
%
%
%
%
%
%
%
%

 Now we can
prove Theorem \ref{mainthm}. Again we should  emphasize that this is inspired by \cite{SS10} with additional restriction on the Hausdorff measure.\\
\begin{proof}[Proof of Theorem \ref{mainthm}.]For $i\ge 1$, we construct a compact
set $A^i\subset[0,1)$ as in Lemma \ref{finitecase} such that $A^i_j:=\underbrace{A^i+\cdots A^i}_j$ has Hausdorff dimension equal to
$\alpha_j$ for $j\le i$, and we will use the same notation as in Lemma \ref{finitecase}, for instance $\gamma_s^i$ and $\eta_s^i$. 
\\For $j\le i$ and $q\in\N^+$, we define the set of integers $\Xi^i_{j,q}$ to be
\begin{equation}\label{Xi}
\bigcup_{k\ge
1}\;[\gamma^i_{\frac{i(i+1)k}2+a^i_{j,q}}\;,\,\eta^i_{\frac{i(i+1)k}2+a^i_{j,q}}-2j\;],
\end{equation}
and for each $p\in \N^+$, we define
\begin{equation}
Q^i_j(p):=\bigcup_{q} \Big\{x\in[0,1):x_s=0\;\;\text{whenever}\;s\le p-2j\;\text{and}
\;s\in\;\Xi^i_{j,q}\Big\}.
\end{equation}
Similar to the proof of Lemma \ref{finitecase}, for any $i\ge j$, there exists an integer $p_{i}$ depending on $i$ only such that when $p>p_{i}$, there exist dyadic intervals $\{I^{r}_{ij}(p)\}_r$ and $q_i(p)$ depending on $p$ and $i$ only such that: $\{I^{r}_{ij}(p)\}_r$ covers $Q_j^i(p)$, and the length of each interval in $\{I^{r}_{ij}(p)\}_r$ is at least $2^{-p}$ and smaller than $2^{-q_i(p)}$, where $\lim\limits_{p\rightarrow\infty}q_i(p)=\infty$ and \begin{equation} \label{intest}
\sum\limits_r |I^{r}_{ij}(p)|^{\alpha_j+\frac 1{2i}}<1.
\end{equation}
Also, denote
\begin{equation}
    T_j^i(p):=\Big\{x\in[0,1):x_s=0\;\text{if}\; s\le p \;\text{and}\;s\in \bigcup_{\substack{k\ge 1,t\ge j \\ q}}
\;[\gamma^i_{\frac{i(i+1)k}2+a_{t,q}^i}\;,\,\eta^i_{\frac{i(i+1)k}2+a_{t,q}^i}\,]\Big\}.
\end{equation}
As discussed in Lemma \ref{lemmahauest}, we can construct a probability measure
$\mu^i_j(p)$ supported on $T^i_j(p)$ such that
\begin{equation}\label{estmeasure}
\mu^i_j(I)\le |I|^{\alpha_j-\frac 1i},
\end{equation}
for any dyadic intervals $I$ with length $\le 2^{-m_{i}}$, where $m_i$ depends on $i$ only.\\
Then we are ready to combine all the previous observations, and construct a rapidly increasing sequence of positive integers $\{\zeta_i\}_{i=1}^{\infty}$ inductively such that 
\begin{equation}\label{restrictions}
\begin{split}
    &\zeta_i>2^i \zeta_{i-1},\\
    &\zeta_i>p_{i},\\
    &\zeta_{i}>(i+1)m_{i+1},\\
    &\zeta_i>\gamma_2^i,\\
    q_i(\zeta_i)>2i&\sum\limits_{j=1}^{i-1}\zeta_j\;(\text{note}\lim\limits_{p\rightarrow\infty}q_i(p)=\infty)
    \end{split}
    \end{equation}and define
$s_i:=\sum\limits_{j=1}^{i-1}\zeta_j$. We construct the compact set $A$ to be
\begin{equation}
\Big\{x\in[0,1):\text{For any } i\ge
1,\;x_{s_i+1}x_{s_i+2}\dots x_{s_{i+1}}=y_1y_2\dots y_{\zeta_i}\text{ for some } y\in A^i\}.
\end{equation}
\\Denote 
\begin{equation}
\begin{split}
S^i:= &\{x\in[0,1): x_{1}x_{2}\dots x_{\zeta_i}=y_1y_2\dots y_{\zeta_i}\text{ for some }
y\in A^i\;\text{and}\;x_s=0 \;\text{for} \;s> \zeta_i\},\\
M^i_j:=&\{x\in[0,1): x_{1}x_{2}\dots
x_{\zeta_i}=y_1y_2\dots y_{\zeta_i}\text{ for some } y\in (\underbrace{S^i+\cdots S^i}_j)\cap[0,1)\;\},\\
N^i_j:=&\{x\in[0,1): x_{1}x_{2}\dots x_{\zeta_i}=y_1y_2\dots y_{\zeta_i}\text{ for some } y\in
(\underbrace{M^i_1+\cdots M^i_1}_j)\cap[0,1)\;\}.
\end{split}
\end{equation}
For any $j$, define $A_j:=\underbrace {A+\cdots A}_j$, then we have
\begin{equation}\label{uplowA_j}
\begin{split}
&A_j\supset\{x\in[0,1): \forall\;i\;\exists \;z\in M^i_j\;\text{such that }x_{s_i+1}x_{s_i+2}\dots x_{s_{i+1}}=z_1z_2\dots z_{\zeta_i}\}
\\
&A_j\subset\{x\in [0,j]: y=\{x \},\forall\; i\;\exists \;z\in N^i_j\;\text{such that } y_{s_i+1}y_{s_i+2}\dots y_{s_{i+1}}=z_1z_2\dots z_{\zeta_i}\}
\end{split}
\end{equation}
Next we will estimate the Hausdorff dimension of $A_j$. Recall the setup in \eqref{Xi}-\eqref{intest}. Notice that $\zeta_i>p_{i}$ and $N^i_j\subset Q^i_j(\zeta_i)$. It follows that there exist dyadic intervals $\{I^{r}_{ij}\}_r$ and $q_i(\zeta_i)$ such that: $\{I^{r}_{ij}\}_r$ covers $Q_j^i(\zeta_i)$, hence covers $N^i_j$; length of each interval $I^r_{ij}$ is at least $2^{-\zeta_i}$ and smaller than $2^{-q_i(\zeta_i)}$; more importantly,
\begin{equation}
\sum\limits_r |I^{r}_{ij}|^{\alpha_j+\frac 1{2i}}<1.
\end{equation}
Thus by \eqref{uplowA_j} and the fact that $q_i(\zeta_i)>2is_i$ \eqref{restrictions}, for $i\ge j$, we can cover $A_j$ by some scale-reduced copies of $\{I^r_{ij}\}_r$ with length smaller than $2^{-2is_i}$, called $\{{J}^u_{ij}\}_u$, such that 
$$\sum\limits_u |J^u_{ij}|^{\alpha_j+\frac 1i}\le 2^{s_i}\cdot 2^{-s_i(\alpha_j+\frac 1i)}\cdot 2^{-\frac {q_i(\zeta_i)}{2i}}<1$$
by using \eqref{restrictions}. Taking a limit $i\rightarrow\infty$, we conclude that $\dim_{\mathcal H} A_j\le \alpha_j$.\\
On the other hand, we construct a Borel probability measure and apply mass distribution principle to get the lower bound on Hausdorff dimension. Recall the measures $\mu_j^i$ supported on $T_j^i(\zeta_i)$ satisfy \eqref{estmeasure}. Note that we have $T_j^i(\zeta_i)\subset M_j^i$ and \eqref{uplowA_j}. If we define a pre-measure $\mu_j$ by \begin{equation}
    \mu_j([x_1x_2\cdots x_s]):=\mu^1_j(\,[x_1x_2\cdots
x_{s_2}]\,)\mu^2_j(\,[x_{s_2+1}\cdots x_{s_3}]\,)\cdots \mu^l_j(\,[x_{s_l+1}\dots x_s]\,),
\end{equation}
where $s_{l} < s\le s_{l+1}$, then $\mu_j$ can be extended to a Borel probability measure supported on $A_j$. \\
For $j\ge 1$, we fix $i>\max\{j,\frac 2{\alpha_j}\}$. For small enough dyadic intervals\\$I=[x_1\cdots x_{s_{l+1}}\cdots x_s]$, where $s_{l+1}\le s<s_{l+2}$: If $s-s_{l+1}<m_{l+1}$, then 
\begin{equation}
\mu_j(I)\le |\,[x_{s_l+1}\cdots x_{s_{l+1}}]\,|^{\alpha_j-\frac 1i}\le c(i)\,|I|^{\alpha_j-\frac 2i}. \label{estformeasure1}
\end{equation}
This inequality follows from the restriction $\zeta_l>(l+1) m_{l+1}$ \eqref{restrictions}. If $s-s_{l+1}\ge m_{l+1}$, then for large $s$ we have  
\begin{equation}
\mu_j(I)\le |\,[x_{s_l+1}\cdots x_s]\,|^{\alpha_j-\frac 1i}\le c(i)\,|I|^{\alpha_j-\frac 1i}.\label{estformeasure2}
\end{equation}  By using mass distribution principle (see e.g. Lemma 1.2.8 in \cite{CY}), we have $\dim_{\mathcal H} A_j\ge \alpha_j-\frac 2i$. Taking $i\rightarrow\infty$, we conclude that $\dim_{\mathcal H} A_j=\alpha_j$ for any $j$.\\\\
Now we estimate the Hausdorff measure of $A$ using Lemma \ref{massdisprin}. Recall the measures $\mu^i$ defined in proof of Lemma \ref{finitecase}, which satisfies \eqref{191004e4.21}:
$$\mu^i(I)\le \frac 1s |I|^{\alpha_1},$$
for every $s \ge 1$ and every $n$-cell $I$ with $n\ge \gamma^i_s$. Moreover, when $n<\gamma^i_1$, for any $n$-cell $I$, $$\mu^i(I)=2^{-n}< 2^{-n\alpha_1}=|I|^{\alpha_1}.$$
Similar to the above proof, we define a Borel probability measure $\mu$ supported on $A$ satisfying
$$\mu([x_1x_2\cdots x_s])=\mu^1(\,[x_1x_2\cdots x_{s_2}]\,)\mu^2(\,[x_{s_2+1}\cdots x_{s_3}]\,)\cdots \mu^l(\,[x_{s_l+1}\dots x_s]\,),$$
for $s_{l} < s\le s_{l+1}$. Then for any $n$-cell $I$ with $n>s_l$, we have 
$$\mu(I)\le 2^{1-l}|I|^{\alpha_1},$$
by using $\zeta_i>\gamma_2^i$ \eqref{restrictions} along with \eqref{191004e4.21}. Note that this inequality can be  generalized to arbitrary measurable sets. By Lemma \ref{massdisprin},
we conclude that $\mathcal H^{\alpha_1}(A)=\infty$. This finishes the construction of the set $A$. 
\end{proof}

\section{Constructing Sets in Higher dimensions}\label{sectionfour}

Suppose for every $s_0\in (0, 1]$, there is a set $E_{s_0}\subset [0, 1)$ with $\dim(E_{s_0})=s_0$ and $\mathcal{H}^{s_0}(E_{s_0})=0$, that is mid-point convex but not convex. For every integer $1\le d_0<d$, we set 
\begin{equation}
E_{s_0+d_0}=E_{s_0}\times [0, 1)^{d_0}.
\end{equation}
It is easy to see that $E_{s_0+d_0}$ has dimension $s_0+d_0$ with $\mathcal{H}^{s_0+d_0}$ measure zero, while being mid-point convex but not essentially convex. \\

Next we consider the problem of constructing a set $E$ with $\mathcal H^s(E)=\infty$ that is mid-point convex but not essentially convex. When $s=s_0+d_0$ is an integer, the construction is trivial. Let us assume that it is not an integer. Suppose that $s_0\in (0, 1)$ and $E_{s_0}\subset [0, 1]$ is the set with $\mathcal H^{s_0}(E_{s_0})=\infty$ that was constructed in the previous section and is mid-point convex but not convex.  In particular, we know that there exists a probability measure supported on $E_{s_0}$ such that 
\begin{equation}\label{mass_infty}
\nu(I)\le c(|I|)|I|^{s_0}, \text{ for all } I,
\end{equation}
with $\lim_{|I|\to 0} c(|I|)=0$.  We define 
\begin{equation}
E_{s_0+d_0}=E_{s_0}\times [0, 1)^{d_0}.
\end{equation}
First of all, the set $E_{s_0+d_0}$ has dimension $s_0+d_0$. Secondly, \eqref{mass_infty} allows us to construct a measure $\bar{\nu}$ supported on $E_{s_0+d_0}$ such that 
\begin{equation}\label{mass_infty2}
\bar{\nu}(B)\le c(|B|)|B|^{s_0}, \text{ for all dyadic rectangles } B,
\end{equation}
with $\lim_{|B|\to 0} c(|B|)=0$. This implies $\mathcal H^{s_0+d_0}(E_{s_0+d_0})=\infty,$ as desired.

		\bibliography{reference}{}
		\bibliographystyle{alpha}

\end{document}